%% file: ZH-split_links.tex
\documentclass[11pt]{amsart}
\usepackage{amssymb}
\usepackage{graphicx, pinlabel}
\usepackage{graphics}
\usepackage{amscd}
\usepackage{enumerate}
\usepackage{comment}
\usepackage[all]{xy}
\usepackage{mathtools} 
\usepackage{xcolor} 
\usepackage{xspace} 
\usepackage[normalem]{ulem} 
\usepackage[margin=3cm, marginparwidth=2.5cm]{geometry}
\usepackage{hyperref}

\newtheorem{thm}{Theorem}[section]
\newtheorem{theorem}[thm]{Theorem}

\newtheorem{lemma}[thm]{Lemma}
\newtheorem{prop}[thm]{Proposition}
\newtheorem{proposition}[thm]{Proposition}

\theoremstyle{definition}

\newtheorem{question}[thm]{Question}
\theoremstyle{remark}

\newtheorem{remark}[thm]{Remark}

\newtheorem{example}[thm]{Example}

\numberwithin{equation}{section}

\renewcommand{\theta}{\vartheta}

\DeclareMathOperator{\im}{im}

\newcommand{\N}{\mathbb{N}}
\newcommand{\Z}{\mathbb{Z}}

\newcommand{\R}{\mathbb{R}}

\newcommand{\F}{\mathbb{F}}

\newcommand{\inv}{^{-1}}

\DeclareMathOperator{\HF}{HF}

\newcommand{\HFp}{\HF^+}
\newcommand{\HFm}{\HF^-}

\renewcommand{\S}{\Sigma}

\definecolor{ao}{rgb}{0.0, 0.5, 0.0} 

\begin{document}

\title[Splitting links by integer homology spheres]{Splitting links by integer homology spheres}

\author{Marco Golla}
\address{CNRS and Laboratoire des math\'ematiques Jean Leray, Nantes Universit\'e}
\email{marco.golla@univ-nantes.fr}

\author{Marco Marengon}
\address{HUN-REN Alfr\'ed R\'enyi Institute of Mathematics}
\email{marengon@renyi.hu}

\date{}

\begin{abstract}
For every $n \geq 3$, we construct 2-component links in $S^{n+1}$ that are split by an integer homology $n$-sphere, but not by $S^n$.
In the special case $n=3$, i.e.\ that of 2-links in $S^4$, we produce an infinite family of links $L_\ell$ and of integer homology spheres $Y_\ell$ such that the link $L_\ell$ is (topologically or smoothly) split by $Y_\ell$ and by no other integer homology sphere in the family.
\end{abstract}

\maketitle

\input{sections/intro.tex}
\input{sections/construction.tex}
\input{sections/all-dimensions.tex}
\input{sections/locally-flat.tex}
\input{sections/smooth.tex}

\input{sections/HF.tex}


\bibliographystyle{alpha}
\bibliography{topology}
\end{document}

%% file: sections/intro.tex

\section{Introduction}

A link in $S^{n+1}$ is the image of an embedding of a disjoint union of $(n-1)$-spheres, which we consider up to ambient isotopy.
We recall that a link is called (\emph{smoothly} or \emph{topologically}) \emph{split} if there is a (smooth or locally-flat) embedding of $S^n$ in $S^{n+1}$ that separates the link, i.e.\ each component of $S^{n+1} \setminus S^n$ contains at least one component of the link.
More generally, given a closed $n$-manifold $Y$, we say that a link in $S^{n+1}$ is (\emph{smoothly} or \emph{topologically}) \emph{split by $Y$} if there exists a (smooth or locally-flat) embedding of $Y$ in $S^{n+1}$ that separates the link.

Moreover, we call links which are (smoothly or topologically) split by an integer homology sphere (smoothly or topologically) \emph{integrally split}.
In this paper we show that there exist integrally split links which are not split by the standard $n$-sphere, answering a question raised (in the case $n=3$) by Peter Kronheimer during Maggie Miller's talk at the Kronheimer birthday conference in Oxford.

\begin{theorem}
\label{t:main}
For every $n\geq3$, there exists a $2$-component link $L$ in $S^{n+1}$ which is smoothly integrally split, but not topologically split by $S^n$.
\end{theorem}

We note that in the case $n=2$ (i.e.\ the case of classical links in $S^3$) the generalised notion of ``split by $Y$'' is not very useful, because every link with at least $2$ components is separated by a torus, which is the simplest possible surface after a 2-sphere. 
On the other hand, such a notion is highly non-trivial whenever $n \geq 3$.
A simple homological argument provides an obstruction to being split by a rational homology sphere: if $L$ is split by a rational homology sphere, then the (single-variable) Alexander polynomial vanishes (see Remark \ref{rem:Alex}).
Homological arguments can also be used to produce examples of triples $(L,Y,Z)$, where $Y$ and $Z$ are rational homology spheres and $L$ is a link split by $Y$ but not by $Z$ (see Remark \ref{rem:QHS}; these examples have $H_1(Y) \neq H_1(Z)$).
Exhibiting such triples where $Y$ and $Z$ are \emph{integer} homology spheres, as in Theorem \ref{t:main}, is more subtle, and requires arguments that go beyond homology.

We also remark (see Remark \ref{rem:homotopy}) that there is no real hope to strengthen Theorem \ref{t:main} by splitting a non-split link by a homotopy sphere: in higher dimensions, there are no embeddings at all of non-trivial homotopy spheres into $S^{n+1}$ \cite{h-cob}, and in lower dimensions the only non-trivial case would be when $n=4$, but the existence of non-trivial homotopy $4$-spheres is still open.

While Theorem \ref{t:main} states the existence of \emph{one} link for every $n \geq 3$, we will actually exhibits \emph{infinitely many} integrally split links that are not split by $S^n$. Moreover, in the case $n=3$ we can produce an infinite family of links and an infinite family of integer homology $3$-spheres such that each link in the family is split by exactly one such integer homology $3$-sphere.

\begin{theorem}
\label{t:infinitelymany-top}
There exist an infinite family $\{L_\ell\}_{\ell \geq 1}$ of $2$-component links in $S^4$ and an infinite family $\{Y_\ell\}_{\ell \geq 1}$ of integer homology $3$-spheres such that $L_k$ is smoothly split by $Y_k$ but not topologically split by any other $Y_\ell$.
\end{theorem}

The proofs of both Theorems \ref{t:main} and \ref{t:infinitelymany-top} is based on a flexible construction of 2-component links $L$ in $S^{n+1}$ which are split by an integer homology $n$-sphere $Y$ such that the double branched cover $\S_2(S^{n+1}, L)$ is a mapping torus on the same $Y$.

In terms of regularity (smooth vs topological), Theorem \ref{t:infinitelymany-top} is already the best possible one (the construction is smooth, the obstruction is topological). Nonetheless, the topological obstruction from Theorem \ref{t:infinitelymany-top} is quite specific to a certain family of integer homology spheres $\{Y_\ell\}_{\ell \in \N}$, so we also prove a smooth obstruction (using Heegaard Floer homology \cite{HF}) which applies more broadly.

\begin{theorem}
\label{t:smooth-obstruction}
Let $L$ be a $2$-component link in $S^4$ which is smoothly split by an integer homology $3$-sphere $Y$, and such that $\S_2(S^{4}, L)$ is a mapping torus on $Y$.
If $L$ is smoothly split also by an integer homology $3$-sphere $Z$, then $\HF_{\mathrm{red}}(Y)$ is a direct summand of $\HF_{\mathrm{red}}(Z)$.
\end{theorem}

We use Theorem \ref{t:smooth-obstruction} to produce another infinite family $\{L'_\ell\}_{\ell \geq1}$ of 2-component links in $S^4$ and another infinite family $\{Y'_\ell\}_{\ell \geq1}$ of $\Z HS^3$ such that $L'_k$ is smoothly split by $Y'_k$ but not by any other $Y'_\ell$ (see Proposition \ref{p:infinitelymany-smooth}).

Ideally we would like to apply the obstruction from Theorem \ref{t:smooth-obstruction} to a case when we know that the link is topologically split, thus detecting exotic splitness. Therefore, we conclude with the following questions.

\begin{question}
Does there exist a $2$-component link in $S^4$ that is topologically but not smoothly (integrally) split?
\end{question}

\begin{question}
Given an integer homology sphere $Y$, does there exist a $2$-component link in $S^4$ that is topologically but not smoothly split by $Y$?
\end{question}

\subsection*{Organisation}
In Section \ref{sec:construction} we present a flexible construction of 2-component links $L$ in $S^{n+1}$ which are smoothly split by an integer homology sphere $Y$. By our construction, the double branched cover $\S_2(S^{n+1}, L)$ is a mapping torus on $Y$.
In Section \ref{sec:alldim} we show that none of such links are split by $S^n$, and we provide infinitely many explicit examples, hence proving Theorem \ref{t:main}.
In Section \ref{sec:top} we prove Theorem \ref{t:infinitelymany-top}, and in Section \ref{sec:smooth} we prove the smooth obstruction (Theorem \ref{t:smooth-obstruction}). We conclude with an appendix recapping some well-known properties of Heegaard Floer homology.

\subsection*{Acknowledgements}
We would like to thank Peter Kronheimer for raising the question and Maggie Miller for the talk on \cite{HKM} which inspired it.
We are also indebted to Michel Boileau for his suggestions and for pointing us to Rong's paper~\cite{Rong}. We thank Clayton McDonald for his comments on a draft of this paper and Danny Ruberman for enlightening conversations about covers.
Finally, we thank the anonymous referee for their insightful comments.
MG thanks the R\'enyi Institute for their hospitality.
MG was partially supported by the \'Etoile Montante project PSyCo of the Region Pays de la Loire.
MM was partially supported by NKFIH grant OTKA K146401.

%% file: sections/construction.tex

\section{The construction}
\label{sec:construction}

In this section we exhibit a construction of integrally split links in every dimension.
We start with the following proposition.

\begin{figure}
    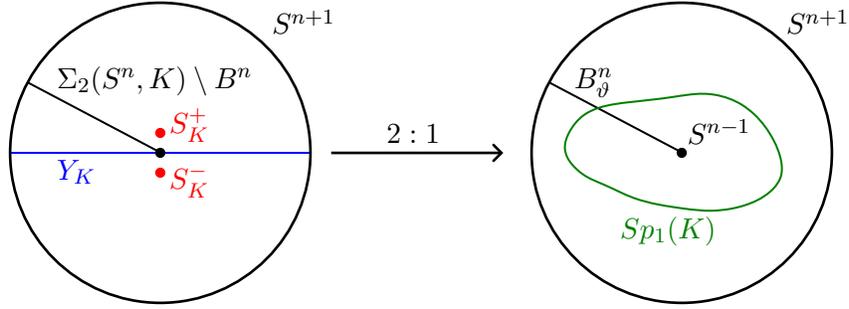
    \caption{The open book decomposition of $S^{n+1}$ with pages diffeomorphic to $B^n$ lifts to an open book decomposition of $\Sigma_2(S^{n+1}, Sp_1(K))$, which is still $S^{n+1}$. The link $L$ is obtained by taking two translates of the binding $S^+_K$ and $S_-(K)$.}
    \label{fig:dbc}
\end{figure}

\begin{proposition}
\label{prop:OBD}
For every knot $K \subset S^n$, there is an open book decomposition of $S^{n+1}$ with binding diffeomorphic to $S^{n-1}$ and pages diffeomorphic to
$
\Sigma_2(S^n,K) \setminus B^n.
$
\end{proposition}

\begin{proof}
Consider the 1-twist spun knot $Sp_1(K) \subset S^{n+1}$.
By the spin construction, there is a (trivial) open book decomposition of $S^{n+1}$ with binding an equatorial $S^{n-1}$ and pages diffeomorphic to $B^n_{\theta}$ such that for all $\theta$ the tangle $Sp_1(K) \cap B^n_\theta$ in $B^n_\theta$ has closure $K$.

By \cite[Corollary 2, p.\ 487]{Z-man}, $Sp_1(K)$ is unknotted in $S^{n+1}$.
Thus, the double cover $\S_2(S^{n+1}, Sp_1(K))$ of $S^{n+1}$ branched over $Sp_1(K)$ is again $S^{n+1}$.
Moreover, we can lift the trivial open book decomposition to the double branched cover, thus obtaining a new open book decomposition of $S^{n-1}$.
Its page $P_\theta$ is the double cover of $B^n_\theta$ branched over (a tangle whose closure is) $K$, i.e.\ $\Sigma_2(S^n,K) \setminus B^n$; its binding is the double cover of $S^{n-1}$ branched over an equatorial $S^{n-3}$, hence an $S^{n-1}$.
\end{proof}

From such an open book decomposition we deduce the existence of an integrally split 2-component link, as in the following proposition.

\begin{proposition}
\label{prop:construction}
For every knot $K$, there exists a $2$-component link $L_K$ in $S^{n+1}$ which is split by the $n$-manifold
\[
Y_K := \Sigma_2(S^n,K) \# \Sigma_2(S^n,m(K)),
\]
and such that $\Sigma_2(S^{n+1}, L_K)$ is a mapping torus on $Y_K$.
\end{proposition}

Here and below, $m(K)$ denotes the mirror of the knot $K$, and therefore $\Sigma_2(S^n,m(K))$ is diffeomorphic to $-\Sigma_2(S^n,K)$.

\begin{proof}
Consider the open book decomposition of $S^{n+1}$ from Proposition \ref{prop:OBD}, and let $S$ denote its binding.
By gluing together two opposite pages $P_\theta$ and $P_{-\theta}$ along their common boundary $S$, we obtain an $n$-manifold
\[
Y_K := \S_2(K \# m(K)) = \S_2(K) \# \S_2(m(K)),
\]
which splits $S^{n+1}$ in two diffeomorphic parts.
That is, we have exhibited a splitting of $S^{n+1}$ by $Y_K$, i.e.\ $S^{n+1} = W^+_K \cup_{Y_K} W^-_K$.
See Figure \ref{fig:dbc}.

We define a link $L_K \subset S^{n+1}$ by taking the disjoint union of two push-offs of the binding $S$, one in each normal direction to $Y_K$. Call these two $(n-1)$-spheres
$S^+_K \subset W^+_K$ and $S^-_K \subset W^-_K$. Note that in this way, by construction, the link is split by $Y_K$.

We now show that the double cover $p \colon \S_2(S^{n+1},L_K) \to S^{n+1}$ of $S^{n+1}$ branched over $L_K$ is a mapping torus on $Y_K$. In the remaining part of the proof we drop the subscript $K$ for ease of notation.

Note that $S^{n+1}$ can be viewed as the union of a regular neighbourhood $N \cong D^2 \times S$ of $S$ containing $L$ (this is possible since each component of $L$ is a push-off of $S$) and the exterior $E$ of $S$, which is a mapping torus on $\S_2(S^n,K) \setminus B^n$.

Since each loop in $E$ winding once around $S$ loops twice around $L$, the restriction of the double cover $p$ to $E$ is the (trivial) disconnected cover. On the other hand, since $N \cong D^2 \times S$, the double cover of $N$ is the product of the identity of $S$ and the double cover of $D^2$ branched over two points, so $p^{-1}(N)$ is the product of $S^{n-1}$ and an annulus.

Thus, $\S_2(S^{n+1}, L)$ is obtained by gluing the two copies of $E$ to the two boundary components of $(I \times S^1) \times S^{n-1}$, in such a way that the $S^{n-1}$-fibration induced from the mapping torus $E$ on its boundary $\partial E$ is identified with the standard $S^{n-1}$-fibration of $S^1 \times S^{n-1}$, viewed in the boundary of $(I \times S^1) \times S^{n-1}$.

It follows that the manifold $\S_2(S^{n+1},L)$ is also a mapping torus. More precisely, since $E$ is the mapping torus of a self-diffeomorphism of $\S_2(S^n,K) \setminus B^n$, it follows that $\S_2(S^{n+1},L)$ is the mapping torus of a self-diffeomorphism of $Y$. 
\end{proof}

From Proposition \ref{prop:construction} we also obtain an obstruction to the link being split by another integer homology sphere $Z$.

\begin{lemma}
\label{lem:obstruction}
Let $L$ be a $2$-component link in $S^{n+1}$ which is (smoothly or topologically) split by an integer homology $n$-sphere $Y$, and such that $\Sigma_2(S^{n+1}, L)$ is a mapping torus on $Y$.

If $L$ is (smoothly or topologically) split by another integer homology $n$-sphere $Z$, then there is a (smooth or topological) invertible cobordism from $Y$ to $Z$.
\end{lemma}

Recall that a cobordism $W$ from $Y$ to $Z$ is called \emph{invertible} if there is another cobordism $W'$ from $Z$ to $Y$ such that the composition $W' \circ W$ is isotopic to $I \times Y$ relative to the boundary. Note that if $Y$ and $Z$ are integer homology spheres, then both $W$ and $W'$ must necessarily be integer homology cobordisms.

\begin{proof}
Suppose that there is a (smooth or topological) embedding of $Z$ in $S^{n+1}$ that splits $L$. 
If we restrict the branched double cover $\S_2(S^{n+1}, L) \to S^{n+1}$ to the complement of $L$, we obtain a regular Abelian cover. Since $H_1(Z;\Z)=0$, we can lift the embedding $Z \hookrightarrow S^{n+1}$ to $\S_2(S^{n+1}, L)$, which is a mapping torus on $Y$, and the fact that $Z$ splits $L$ implies that $Z$ lifts as a homological section of $\S_2(S^{n+1}, L)$. 

The long exact sequence in homotopy associated with the fibration $Y \to \Sigma_2(S^{n+1},L) \to S^1$ shows that the image of $\pi_1(Y)$ inside $\pi_1(\Sigma_2(S^{n+1},L))$ is the commutator subgroup, hence the map
\[
\R \times Y \to \Sigma_2(S^{n+1},L)
\]
obtained by unwinding the mapping torus is in fact the maximal Abelian cover of $\S_2(S^{n+1},L)$. Again, since $H_1(Z;\Z)=0$, we can lift $Z$ to $\R \times Y$, again as a homological section. By choosing two parallel lifts of $Y$ on either sides of $Z$ we obtain an identity cobordism $I \times Y$ which is split by $Z$. One of the two components of $(I \times Y) \setminus Z$ gives a cobordism from $Y$ to $Z$, and the other part is the left inverse cobordism.
\end{proof}

\begin{remark}
\label{rem:QHS}
With some homological observations, the arguments in this section can be used to produce a link $L$ and two rational homology spheres $Y$ and $Z$ such that $L$ is split by $Y$ but not by $Z$.
First use Proposition \ref{prop:construction} on a knot $K$ with $\det K \neq 1$ to produce a link $L = L_K$ split by a rational homology sphere $Y = Y_K$ with $H_1(Y;\Z) \neq 0$. Now let $Z$ be another rational homology sphere, with $H_1(Z;\Z)$ of order coprime with that of $H_1(Y;\Z)$, and assume by contradiction that $Z$ splits $L$, too.

We can then follow the argument of Lemma \ref{lem:obstruction}. First one shows that $Z$ lifts from $S^n \setminus L$ to $\Sigma_2(L) \setminus \pi^{-1}(L)$, using the fact that the two-fold cover is Abelian, and that the torsion $\Z$-module $H_1(Z)$ must map trivially to the free $\Z$-module $H_1(S^n \setminus L)$. Then, the fact that $\Sigma_2(L)$ is a mapping torus over $Y$ implies via a Mayer--Vietoris sequence that $H_1(\Sigma_2(L)) \cong \Z \oplus Q$, where $Q$ is a quotient of $H_1(Y)$.
Thus, the map $H_1(Z) \to H_1(\Sigma_2(L))$ vanishes, and thus $Z$ lifts to the Abelian cover $I \times Y$. Moreover, $Z$ is homologically primitive in $I \times Y$, by the same argument as in Lemma \ref{lem:obstruction}, so $[Y] = \pm[Z]$ in $H_3(I \times Y) \cong \Z$. Thus, restricting the projection gives a degree-one map from $Z$ to $Y$, which must induce a surjection in homology, and this gives the desired contradiction.
\end{remark}

\begin{remark}
\label{rem:Alex}
If one wants to construct links that are not split by any rational homology sphere, then there is a simple obstruction coming from the single-variable Alexander polynomial: if $L$ is split by $Y$ and $b_1(Y) = 0$, then $\Delta_L(t) = 0$.
To see this, note that $Y$ lifts to the infinite cyclic cover $\tilde X$ of $L$ associated to the morphism that sends each meridian $\mu_i$ of $L$ to a generator of $\Z$ (here we are taking $L$ with an orientation), and so the preimage of $Y$ to $\tilde X$ is $\coprod_{m\in\Z} t^m \, Y_0$. Consider a loop $\gamma$ in $S^4$ in the homology class $[\mu_1] - [\mu_2]$ that intersects $Y$ geometrically twice.
See Figure \ref{fig:infinitecyclic}.
$\gamma$ also lifts to a loop $\gamma_0$ in $\tilde X$, and the pairing satisfies:
\[
\langle Y_0, \gamma_0 \rangle_{\Z[t^{\pm1}]} = \sum_{m \in \Z} t^{m} \langle t^{-m} Y_0, \gamma_0 \rangle = \pm (t-1).
\]
Thus, the $\Z[t^{\pm1}]$-module $H_1(\tilde X; \Z)$ is not torsion, and therefore its first Alexander ideal vanishes.
\end{remark}

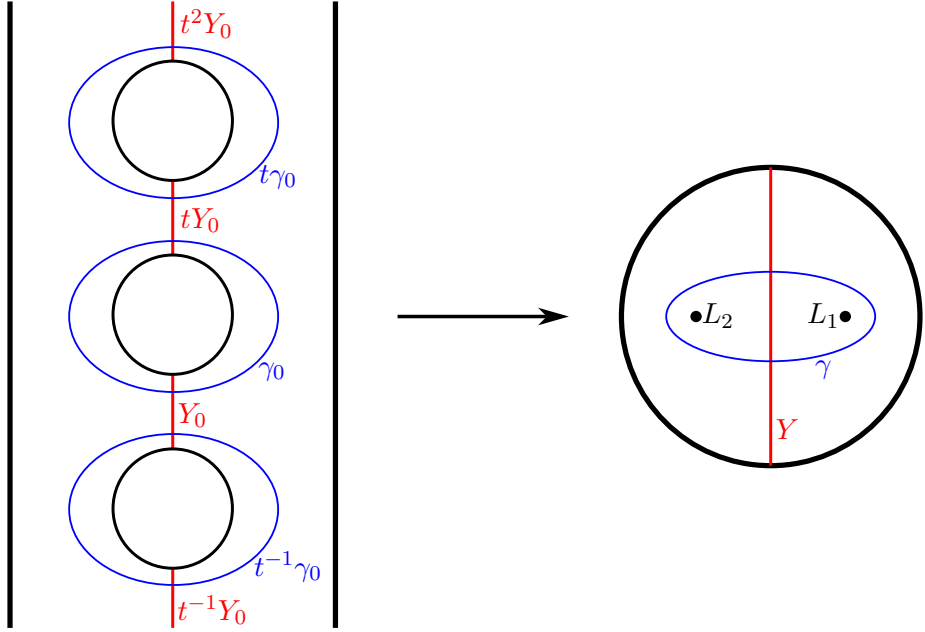
\begin{figure}
\labellist
\pinlabel $L_1$ at 410 157
\pinlabel {\color{red}$Y$} at 392 100
\pinlabel {\color{blue}$\gamma$} at 410 130
\pinlabel {\color{red}$Y_0$} at 93 107
\pinlabel {\color{red}$tY_0$} at 97 205
\pinlabel {\color{red}$t^{-1} Y_0$} at 103 9
\pinlabel {\color{red}$t^{2} Y_0$} at 100 304
\pinlabel {\color{blue}$\gamma_0$} at 133 129
\pinlabel {\color{blue}$t\gamma_0$} at 136 227
\pinlabel {\color{blue}$t^{-1} \gamma_0$} at 141 31
\pinlabel {\color{black}$L_2$} at 357 157
\endlabellist
\includegraphics[scale=0.75]{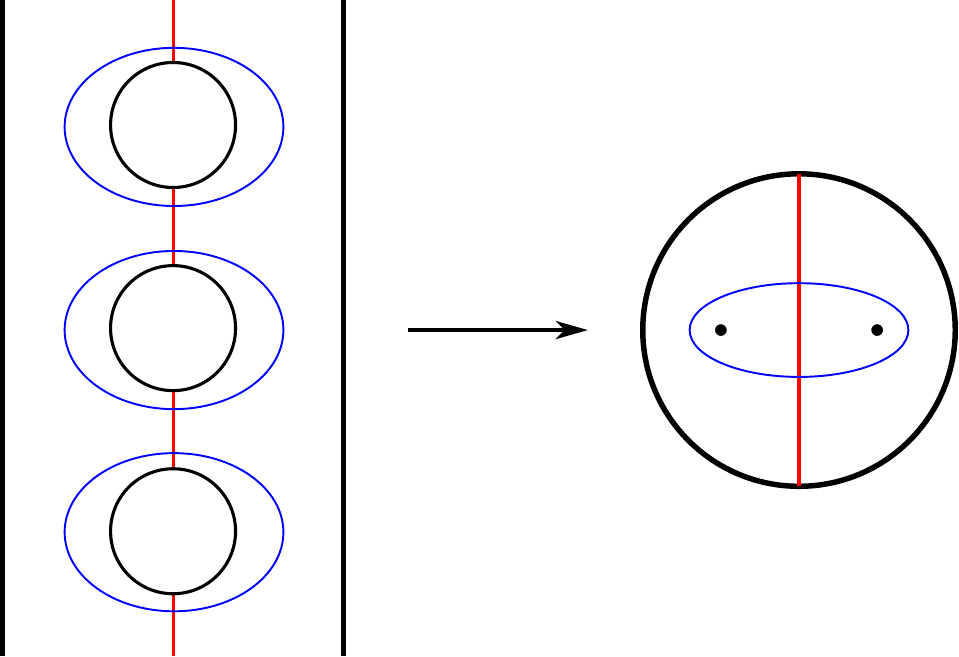}
\caption{Schematic of the infinite cyclic cover of $S^n$ branched over the link $L$.}
\label{fig:infinitecyclic}
\end{figure}

%% file: figures/dbc.pdf_tex
\begingroup%
  \makeatletter%
  \providecommand\color[2][]{%
    \errmessage{(Inkscape) Color is used for the text in Inkscape, but the package 'color.sty' is not loaded}%
    \renewcommand\color[2][]{}%
  }%
  \providecommand\transparent[1]{%
    \errmessage{(Inkscape) Transparency is used (non-zero) for the text in Inkscape, but the package 'transparent.sty' is not loaded}%
    \renewcommand\transparent[1]{}%
  }%
  \providecommand\rotatebox[2]{#2}%
  \newcommand*\fsize{\dimexpr\f@size pt\relax}%
  \newcommand*\lineheight[1]{\fontsize{\fsize}{#1\fsize}\selectfont}%
  \ifx\svgwidth\undefined%
    \setlength{\unitlength}{310.66969984bp}%
    \ifx\svgscale\undefined%
      \relax%
    \else%
      \setlength{\unitlength}{\unitlength * \real{\svgscale}}%
    \fi%
  \else%
    \setlength{\unitlength}{\svgwidth}%
  \fi%
  \global\let\svgwidth\undefined%
  \global\let\svgscale\undefined%
  \makeatother%
  \begin{picture}(1,0.36749588)%
    \lineheight{1}%
    \setlength\tabcolsep{0pt}%
    \put(0,0){\includegraphics[width=\unitlength,page=1]{figures/dbc.pdf}}%
    \put(0.31937983,0.32817944){\color[rgb]{0,0,0}\makebox(0,0)[lt]{\lineheight{1.25}\smash{\begin{tabular}[t]{l}$S^{n+1}$\end{tabular}}}}%
    \put(0.82325911,0.19501821){\color[rgb]{0,0,0}\makebox(0,0)[lt]{\lineheight{1.25}\smash{\begin{tabular}[t]{l}$S^{n-1}$\end{tabular}}}}%
    \put(0.94301062,0.32817943){\color[rgb]{0,0,0}\makebox(0,0)[lt]{\lineheight{1.25}\smash{\begin{tabular}[t]{l}$S^{n+1}$\end{tabular}}}}%
    \put(0.19481933,0.2079579){\color[rgb]{0,0,0}\makebox(0,0)[lt]{\lineheight{1.25}\smash{\begin{tabular}[t]{l}{\color{red}$S_K^+$}\end{tabular}}}}%
    \put(0.19481933,0.14036172){\color[rgb]{0,0,0}\makebox(0,0)[lt]{\lineheight{1.25}\smash{\begin{tabular}[t]{l}{\color{red}$S_K^-$}\end{tabular}}}}%
    \put(0.68606357,0.26240029){\color[rgb]{0,0,0}\makebox(0,0)[lt]{\lineheight{1.25}\smash{\begin{tabular}[t]{l}$B_\theta^{n}$\end{tabular}}}}%
    \put(0.05838591,0.26240029){\color[rgb]{0,0,0}\makebox(0,0)[lt]{\lineheight{1.25}\smash{\begin{tabular}[t]{l}$\Sigma_2(S^n,K) \setminus B^{n}$\end{tabular}}}}%
    \put(0.05838591,0.15134938){\color[rgb]{0,0,0}\makebox(0,0)[lt]{\lineheight{1.25}\smash{\begin{tabular}[t]{l}{\color{blue}$Y_K$}\end{tabular}}}}%
    \put(0.74174394,0.0804516){\color[rgb]{0,0,0}\makebox(0,0)[lt]{\lineheight{1.25}\smash{\begin{tabular}[t]{l}{\color{ao}$Sp_1(K)$}\end{tabular}}}}%
    \put(0.45846564,0.19527655){\color[rgb]{0,0,0}\makebox(0,0)[lt]{\lineheight{1.25}\smash{\begin{tabular}[t]{l}$2:1$\end{tabular}}}}%
  \end{picture}%
\endgroup%

%% file: sections/all-dimensions.tex

\section{Higher-dimensional links}
\label{sec:alldim}

\begin{proposition}
\label{prop:alldim}
Let $L$ be a $2$-component link in $S^{n+1}$ which is split by an integer homology $n$-sphere $Y$ with $\pi_1(Y) \neq 1$, and such that $\Sigma_2(S^{n+1}, L)$ is a mapping torus on $Y$.
Then $L$ is not topologically split by $S^{n}$.
\end{proposition}

\begin{proof}
By Lemma \ref{lem:obstruction} there is an invertible topological cobordism from $Y$ to $S^n$, or, equivalently, there is a locally-flat embedding of $S^n$ in $\R \times Y$ as a homological section. Composing with the projection, this would give a degree-1 map from $S^{n}$ to $Y$. A degree-1 map always induces an epimorphism on $\pi_1$, but this is impossible under the assumption that $\pi_1(Y) \neq 1$.
\end{proof}

\begin{remark}
\label{rem:homotopy}
Due to the presence of the hypothesis $\pi_1(Y)\neq1$ in the statement of Proposition \ref{prop:alldim}, the reader may wonder if it is possible to find a 2-component link in $S^{n+1}$ which is split by an $n$-dimensional \emph{homotopy} sphere, but not by $S^n$.
The answer is always negative.

When $n \neq 4$ it is impossible to even embed an $n$-dimensional exotic sphere $Y$ in $S^{n+1}$. Indeed, an embedding $Y \hookrightarrow S^{n+1}$ would give an $h$-cobordism between $Y$ and $S^n$ (after removing a small ball from $S^{n+1}$), hence showing that $Y$ is diffeomorphic to $S^n$ \cite{h-cob}. (Note that $n \geq 7$ because $Y$ is assumed to be exotic.)

When $n = 4$, even if the existence of exotic 4-spheres is open, one can still show that a homotopy sphere $Y^4$ splits a link $L$ if and only if the standard sphere $S^4$ does. 
This can be deduced from the fact that there is a unique smooth $h$-cobordism between any two homotopy 4-spheres $Y$ and $Y'$, up to diffeomorphism rel boundary \cite{Lawson, Kreck}.
Given the unique $h$-cobordism $W \colon Y \to Y'$ and its inverse $-W \colon Y' \to Y$, the composition $(-W) \circ W$ must be diffeomorphic to $I \times Y$, by uniqueness.
It follows that $Y'$ sits in $I \times Y$ as a homological section.
In other words, a regular neighbourhood $\nu(Y)$ of a homotopy 4-sphere $Y\subset S^5$ contains every other homotopy 4-sphere $Y'$ as a homological section, so $L$ is split by a homotopy 4-sphere if and only if it is split by all of them.
\end{remark}

\begin{theorem}
For every $n\geq3$, there exists a $2$-component link $L$ in $S^{n+1}$ which is smoothly split by an integer homology $n$-sphere $Y$ but is not topologically split by $S^n$.
\end{theorem}

\begin{proof}
Let $J$ be a non-trivial knot in $S^3$ with $\det J = 1$, so that $\S_2(S^3, J)$ is an integer homology sphere with $\pi_1(\S_2(S^3, J)) \neq 1$.
(Concretely, we can choose an infinite family of such knots by taking $J$ to be any torus knot $T(p,q)$ with $p,q > 1$ both odd, in which case $\S_2(S^3, J)$ is the Brieskorn sphere $\S(2,p,q)$.)

We define a knot $K$ in $S^{n}$ by taking the $0$-spun knot $n-3$ times:
\[
K := Sp_0 \left( \cdots \left( Sp_0(J) \right) \cdots \right).
\]
From the $0$-spinning construction, it is straightforward to check that if $J$ is a knot whose double branched cover is an integer homology sphere, then so is $Sp_0(J)$. Thus, iterating this reasoning, we have that $\S_2(S^n, K)$ is an integer homology sphere.
Moreover, since $0$-spinning preserves the knot group, we have that $\pi_1(S^{n} \setminus K) \cong \pi_1(S^3 \setminus J)$, and therefore $\pi_1(\S_2(S^n, K)) \cong \pi_1(\S_2(S^3, J))$, being their unique index-$2$ subgroups: thus, $\S_2(S^n, K)$ is also an integer homology sphere with $\pi_1(\S_2(S^n, K)) \neq 1$.

By Proposition \ref{prop:construction}, there exists a link $L_K$ in $S^{n+1}$ which is split by the manifold $Y_K = \S_2(S^n, K) \# \S_2(S^n, m(K))$ (which by the previous paragraph is an integer homology sphere with $\pi_1(Y_K) \neq 1$), and such that $\S_2(S^{n+1}, L)$ is a mapping torus over $Y_K$.
It then follows from Proposition \ref{prop:alldim} that $L_K$ is not topologically split by $S^n$.
\end{proof}

%% file: sections/locally-flat.tex

\section{Locally-flatly splitting in dimension $4$}
\label{sec:top}

In this section we prove the following more precise version of Theorem \ref{t:infinitelymany-top}.

\begin{prop}\label{p:infinitelymany}
For a positive torus knot $T = T(p,q) \neq T(3,5)$ with $p,q$ both odd, let $L_T$ and $Y_T$ denote the link and the integer homology spheres given by Proposition \ref{prop:construction}. If $T, T'$ are two positive torus knots as above, then $L_T$ is smoothly split by $Y_T$ but not topologically split by $Y_{T'}$.
\end{prop}

We are indebted to Michel Boileau for suggesting the idea of the proof.

Recall that if $T = T(p,q)$ is a torus knot with $p$, $q$ both odd, then $\det T = 1$ and the double cover $\Sigma_2(S^3,T)$ is the Brieskorn sphere $\Sigma(2,p,q)$, which is a small Seifert fibred space. If, furthermore, $T \neq T(3,5)$, then $\pi_1(\Sigma(2,p,q))$ is infinite. Since small Seifert fibred spaces are irreducible, by the sphere theorem the universal cover of $\Sigma(2,p,q)$ has trivial $\pi_2$. Since the fundamental group of $\Sigma(2,p,q)$ is infinite, the universal cover has trivial $H_3$, and by the Hurewicz theorem $\Sigma(2,p,q)$ has trivial $\pi_3$ too.

Moreover, there are no incompressible surfaces in $\Sigma(2,p,q)$, and the only connected incompressible surface in $\Sigma(2,p,q) \# \overline{\Sigma(2,p,q)}$ is the separating 2-sphere. It follows that the only incompressible surfaces in $\Sigma(2,p,q) \# \overline{\Sigma(2,p,q)}$ are parallel copies of the separating 2-sphere. Indeed, any incompressible surface lives in one of the summands (by an innermost circle argument on the separating 2-sphere).

As the reader might have anticipated, we will be interested in looking at degree-$1$ maps between connected sums of Brieskorn spheres. Rong has studied degree-1 maps between Seifert manifolds in~\cite{Rong} and proved such maps are not so common. In particular,~\cite[Theorem~3.2]{Rong} asserts that if there exists a degree-$1$ map between two Brieskorn spheres $\Sigma(2,p,q)$ and $\Sigma(2,p',q')$ with infinite fundamental group (with either orientation), then\footnote{This is easy to verify using the condition on the least common multiples of the multiplicities of the fibres in the definition of the partial order.} $\{p,q\} = \{p',q'\}$.

\begin{proof}
Suppose that there are $T=T(p,q)$ and $T'=(p',q')$ positive torus knots with $p,q,p',q'$ all odd such that $L_{T}$ is split by $Y_{T'}$. For convenience, call $Y = \Sigma(2,p',q')$ and $Z = \Sigma(2,p,q)$, so that $Y_{T'} \cong Y\#\overline Y$ and $Y_{T} \cong Z\# \overline Z$.

By Lemma \ref{lem:obstruction} there is an invertible cobordism from $Z\# \overline Z$ to $Y\# \overline Y$, and therefore a topological embedding of $Y\#\overline{Y}$ in $\R \times (Z\# \overline{Z})$ such that post-composing with the projection $\R \times (Z\# \overline{Z})\to Z\# \overline{Z}$ yields a continuous degree-$1$ map
\[
f\colon Y\#\overline{Y} \to Z\# \overline{Z}.
\]
We want to show that such an $f$ cannot exist. Up to perturbations, we can suppose that $f$ is smooth.

Call $S$ the sphere that splits $Z\# \overline{Z}$ into $Z\setminus B^3$ and $\overline{Z}\setminus B^3$. Up to a small perturbation, we can suppose that $f$ is transverse to $S$, and that in particular $f\inv(S)$ is a surface in $Y\#\overline Y$. Up to performing some embedded surgeries, as explained in~\cite[Section~I.2.3]{Laudenbach}, we can suppose that $f\inv(S)$ is a union of incompressible surfaces.

Since $Y$ is a small Seifert fibred space, the only connected incompressible surface in $Y\#\overline Y$ is the separating sphere, so $f\inv(S)$ is a union of parallel spheres. Suppose that there are at least two spheres in $f\inv(S)$, and choose two consecutive ones, which cobound a 3-manifold $A$, with $A \cong S^2\times [0,1]$. The restriction $f_A$ of $f$ to $A$ maps $A$ either into $Z \setminus B^3$ or into $\overline Z\setminus B^3$ (by connectedness). Suppose that we are in the former case (the latter being entirely equivalent). Note that $f(\partial A) \subset \partial (Z\setminus B^3)$, so we can extend $f_A$ to a map $g_A \colon S^3 \to Z$, where we view $S^3$ as obtained by filling the boundary of $A$ by two 3-balls.

Since $Z$ is a Seifert fibred space with infinite fundamental group, $\pi_3(Z) = 0$, and therefore $g_A$ is null-homotopic, and in particular it has degree $0$. It follows that also $f_A$ has degree $0$ (since $A$ is the preimage of $Z\setminus B^3 \subset Z$, and we can compute the degree by counting the points in the preimage of a generic point).

Since the map $f$ has degree $1$, and its restriction to all these copies of $S^2\times[0,1]$ has degree $0$, the preimage of $Z\setminus B^3$ cannot be contained in a union of these copies, and in particular it must contain, say, $Y\setminus B^3 \subset Y\#\overline Y$. It follows that the restriction to $Y\setminus B^3$ has degree $1$ onto $Z\setminus B^3$. We can now extend the restriction of $f$ to $Y \setminus B^3$ to a degree-$1$ map $g \colon Y \to Z$, but no such map exists by work of Rong~\cite[Theorem~3.2]{Rong}, as explained above.

This concludes the proof.
\end{proof}

%% file: sections/smooth.tex

\section{Smoothly splitting in dimension $4$}
\label{sec:smooth}

Since the proof of Proposition \ref{p:infinitelymany} uses strong geometric assumptions on the manifolds $Y_{K}$, we also prove a smooth obstruction, namely Theorem \ref{t:smooth-obstruction} from the introduction.
In Theorem \ref{t:smooth-obstruction} we use the Heegaard Floer homology package. We recall the notation and the properties we will need in Appendix \ref{app:HF}.

\begin{proof}[{Proof of Theorem \ref{t:smooth-obstruction}}]
Lemma \ref{lem:obstruction} implies that there is an invertible cobordism $W$ from $Y$ to $Z$, with inverse $W'$.
Since $W$ and $W'$ are integer homology cobordisms, the cobordism maps $F_{W}$ and $F_{W'}$ in $\HF_{\mathrm{red}}$ preserve the Maslov grading. 

By functoriality, the composition
\[
\HF_{\mathrm{red}}(Y) \xrightarrow{F_{W}} \HF_{\mathrm{red}}(Z) \xrightarrow{F_{W'}} \HF_{\mathrm{red}}(Y)
\]
is the identity map, so we have a splitting of graded $\F[U]$-modules
\[
\HF_{\mathrm{red}}(Z) \cong \ker F_{W'} \oplus \im F_{W},
\]
and $\im F_{W} \cong \HF_{\mathrm{red}}(Y)$, as desired.

This argument has appeared for invertible concordances in \cite{concordance}. A similar argument for ribbon concordances and cobordisms received much attention in last years, sparked by \cite{Zemke} (see \cite{DLVVW} for the case of ribbon homology cobordisms). We were also inspired by~\cite{LLP}.
\end{proof}

We apply Theorem \ref{t:smooth-obstruction} to produce a family of examples (in the smooth category) which is not covered by Proposition \ref{p:infinitelymany}.

\begin{prop}
\label{p:infinitelymany-smooth}
For each $\ell \ge 2$, let $K_\ell = T(4\ell+1, 8\ell+3)\#T(3,5)$, and let $L_{K_\ell}$ and $Y_{K_\ell}$ denote the links and the integer homology spheres given by Proposition \ref{prop:construction}.

For each $m, n \ge 2$, the link $L_{K_n}$ in $S^4$ is smoothly split by $Y_{K_n}$ but not by $Y_{K_m}$ for any $m \neq n$.
\end{prop}

The following definition is a compact way of referring to integer homology spheres with $\HF_{\rm red}$ of a special form, which is convenient for the set of examples that we will study here.
We say that a homology 3-sphere $Y$ is \emph{of size} $\ell \ge 2$ if:
\begin{enumerate}
\item \label{it:length} the length of the longest finite tower in $\HFm_{\mathrm{even}}(Y)$ is $\ell$;
\item \label{it:1-towers} there exists at least one finite tower of length $1$ in positive degree in $\HFm_{\mathrm{even}}(Y)$, and the smallest positive degree that supports a length-1 tower is $2(\ell-1)$.
\end{enumerate}
Note that not all homology 3-spheres have a well-defined size.

\begin{remark}
\label{rem:sizes}
If $Y_m$ and $Y_n$ are integer homology spheres of different sizes $m$ and $n$ respectively, then neither $\HF_{\rm red}(Y_m)$ nor $\HF_{\rm red}(Y_n)$ is a direct summand of the other. Suppose without loss of generality that $m>n$: then property \eqref{it:length} obstructs $\HF_{\rm red}(Y_m)$ from being a direct summand of $\HF_{\rm red}(Y_n)$, while property \eqref{it:1-towers} obstructs the other inclusion.
\end{remark}

\begin{example}
\label{ex:sizes}
Call $\Sigma_\ell = \Sigma(2,4\ell+1,8\ell+3)$.
We want to show that $Y_\ell = \Sigma_\ell \# \overline{\Sigma}_\ell$ has size $\ell$ for each $\ell \ge 2$.

Rustamov computed $\HFp(\overline{\Sigma}_\ell)$ for every $\ell$~\cite[Theorem~1.1]{Rustamov}, and from his computation we can immediately deduce $\HFm(\Sigma_\ell)$ using \cite[Proposition~2.5]{HF:p&a}. Specifically, we have:
\begin{itemize}
\item $d(\Sigma_\ell) = 0$;
\item $\HFm(\Sigma_\ell)$ is entirely supported in non-positive even degree;
\item there are exactly two towers of length $1$ in $\HFm(\Sigma_\ell)$, supported in Maslov gradings $-\ell(\ell+1)$ and $-(\ell-1)\ell$;
\item there is exactly one tower of length $\ell$ in $\HFm(\Sigma_\ell)$, which is the maximal possible length among finite towers;
\item the closest grading supporting the generator of a finite tower in $\HFm(\Sigma_\ell)$ to $-\ell(\ell+1)$ or $-(\ell-1)\ell$ is $-(\ell-2)(\ell-1)$.
\end{itemize}

Now, both properties~\eqref{it:length} and~\eqref{it:1-towers} above for $Y_\ell = \Sigma_\ell \# \overline{\Sigma}_\ell$ follow from~\cite[Corollary 6.3]{HF:p&a} and~\cite[Proposition~2.5]{HF:p&a}.

Indeed, since $\HFm_{\rm red}(\Sigma_\ell)$ is supported in even degree, we have that $\HFm_{\rm red}(\overline{\Sigma}_\ell)$ is supported in odd degree (see Equation \eqref{eq:reddual}). By the K\"unneth formula \cite[Corollary~6.3]{HF:p&a} we can compute $\HFm_{\mathrm{red}}(Y_\ell)$ (see Equation \eqref{eq:redKunneth}), and by taking the even-degree summand we obtain
\begin{equation}
\label{eq:HFredYell}
\HFm_{\mathrm{red, even}}(Y_\ell) \cong \left( \HFm_{\mathrm{red}}(\Sigma_\ell) \otimes \left(\HFm_{\mathrm{red}}(\Sigma_\ell)\right)^* \right) \oplus \HFm_{\mathrm{red}}(\Sigma_\ell).
\end{equation}
(Note that there is a degree shift by $+1$ in $\HF_\textrm{red}$ when changing the orientation, and a degree shift by $-1$ when applying the Tor functor in the tensor product, and they cancel out.)

Since
\begin{equation}
\label{eq:tensor}
\mathbb{F}[U]/(U^k) \otimes_{\mathbb{F}[U]} \mathbb{F}[U]/(U^j) \cong \mathbb{F}[U]/(U^{\min\{k,j\}}),
\end{equation}
the length of the longest tower in $\HFm_{\mathrm{red}}(\Sigma_\ell)$ will be preserved in $\HFm_{\mathrm{red, even}}(Y_\ell)$ too, hence showing property \eqref{it:length}.

We now turn to property \eqref{it:1-towers}. Referring to Equation \eqref{eq:HFredYell}, towers of length 1 with positive Maslov degree can only come from the first summand (i.e., the tensor product), because $\HFm_{\mathrm{red}}(\Sigma_\ell)$ is supported in non-positive degree. On the other hand, the first summand is symmetric, so it is enough to find the smallest Maslov grading in absolute value that supports a tower of length $1$.

By Equation \eqref{eq:tensor}, towers of length $1$ in the first summand are obtained by tensoring a tower of length $1$ with a tower of any length, and their Maslov degree is the distance (in Maslov degree) between the generators of those two towers. From the properties above, the minimum such distance in absolute value is $(\ell-1)\ell -(\ell-2)(\ell-1) = 2(\ell-1)$, thus proving property \eqref{it:1-towers} for $Y_\ell$.

Thus, $Y_\ell$ is of size $\ell$ for each $\ell \ge 2$. Since the Poincar\'e homology sphere is an L-space, we also have that $Y_\ell \# n(P \# \overline P)$ has size $\ell$ for each $\ell \ge 2$.
\end{example}

We can now prove Proposition~\ref{p:infinitelymany-smooth}.

\begin{proof}[Proof of Proposition~\ref{p:infinitelymany-smooth}]
Choose $m, n$ distinct integers and suppose that the link $L_{K_n}$ in $S^4$ is smoothly split by $Y_{K_m}$.
Then, by Theorem \ref{t:smooth-obstruction}, $\HF_{\rm red}(Y_{K_n})$ is a direct summand of $\HF_{\rm red}(Y_{K_m})$.
However, $Y_{K_n}$ is of size $n$ and $Y_{K_m}$ is of size $m$, by Example \ref{ex:sizes}, so as observed in Remark \ref{rem:sizes} neither of their $\HF_{\rm red}$ can be a direct summand of the other.
\end{proof}

%% file: sections/HF.tex
\appendix

\section{Heegaard Floer homology}
\label{app:HF}
We recall here the basic facts we need about Heegaard Floer homology. This was an invariant defined by Ozsv\'ath and Szab\'o in~\cite{HF, HF:p&a} that associates to a 3-manifold a collection of groups and to cobordisms between them homomorphisms between the corresponding groups. Given the context of the paper, we will restrict to 3-manifolds that are integer homology spheres and to 4-manifold cobordisms that are rational homology cobordisms. This will greatly simplify the discussion and allow us to avoid spin$^c$ structures and rational gradings. We will work over the field $\F$ with two elements.

If $Y$ is an integer homology sphere, one of the groups associated to $Y$ is called $\HFm(Y)$: this is a finitely generated $\Z$-graded $\F[U]$-module, where $U$ is a formal variable of degree $-2$. The degree is called the \emph{Maslov grading}. It sits into an exact sequence of $U$-equivariant degree-$0$ maps\footnote{A word of caution here. There is also an $\HFp_{\mathrm{red}}$, which is a quotient of $\HFp$, and it differs from $\HFm$ by a grading shift: $\HFm_{\rm red} \cong \HFp_{\rm red}[+1]$ as graded $\F[U]$-modules. When we use $\HF_{\rm red}$ we refer to either of them, but we assume that the choice we make is coherent.}:
\[
0 \to \HF^-_{\rm red}(Y) \to \HFm(Y) \to \mathcal{T}^-_{d(Y)} \to 0,
\]
where $\mathcal{T}^-$ is an \emph{infinite tower}, $\mathcal{T}^- \cong \F[U]$, $d(Y)$ is an even integer and $1 \in \mathcal{T}^-_{d(Y)}$ sits in degree $d(Y)$, and $\HF^-_{\rm red}(Y)$ is a finitely-generated $\F[U]$-module that is finite-dimensional over $\F$.
In particular, $\HFm_{\rm red}(Y)$ splits as a direct sum of \emph{finite towers}:
\[
\HFm_{\rm red}(Y) \cong \bigoplus_i (\F[U]/(U^{\ell_i}))_{d_i},
\]
for some (not necessarily even) integers $d_i$ and some positive integers $\ell_i$. Here the $d_i$ signifies that the grading of $1 \in \F[U]/(U^{\ell_i})$ is $d_i$. (This is also the highest degree among all elements.) The quantity $d(Y)$ is called the $d$-\emph{invariant} or \emph{correction term} of $Y$.

We denote with $\HFm_{\mathrm{even}}(Y)$ (resp.\ $\HFm_{\mathrm{red, even}}(Y)$) the even-degree summand of $\HFm(Y)$ (resp.\ $\HFm_{\mathrm{red}}(Y)$). Since $U$ has degree $-2$, $\HFm_{\mathrm{even}}(Y)$ is an $\F[U]$-submodule of $\HFm(Y)$.

There is a formula relating $\HFm(Y)$ and $\HFm(\overline Y)$: $d(\overline Y) = -d(Y)$ and
\begin{equation}
\label{eq:reddual}
\HFm_{\mathrm{red}}(\overline Y) = (\HFm_{\mathrm{red}}(Y))^*[+1],
\end{equation}
where the $*$ indicates that the degrees and the $U$-action are reversed, and the $[+1]$ indicates that there's a degree shift (this follows from \cite[Proposition~2.5]{HF:p&a} and the fact that $\HFm_{\rm red} \cong \HFp_{\rm red}[+1]$). So, for instance, to a finite tower $(\F[U]/(U^{\ell_i}))_{d_i}$ in $\HFm(Y)$ there corresponds a finite tower $(\F[U]/(U^{\ell_i}))_{1-2\ell_i-d_i}$ in $\HFm(\overline Y)$.

There is a K\"unneth formula relating $\HFm(Y), \HFm(Y')$, and $\HFm(Y\# Y')$: if we denote $H = \HFm_\mathrm{red}(Y)$ and $H' = \HFm_\mathrm{red}(Y')$, then by \cite[Corollary~6.3]{HF:p&a}
\begin{equation}
\label{eq:redKunneth}
\HFm_\mathrm{red}(Y\#Y') \cong H[d(Y')] \oplus H'[d(Y)] \oplus \left(H \otimes H'\right) \oplus \left(H\otimes H'\right)[-1].
\end{equation}

If $W$ is a rational homology cobordism from $Y$ to $Y'$, then there is an induced homomorphism of $\F[U]$-modules
\[
F^-_W \colon \HFm(Y) \to \HFm(Y')
\]
which has degree $0$, which descends to an analogous homomorphism in $\HFm_{\mathrm{red}}$. This association is functorial: if $W'$ is another rational homology cobordism from $Y'$ to $Y''$, then the composition of the associated maps is the map associated to the cobordism $W''$ obtained by stacking $W$ and $W'$. In other words,
\[
F^-_{W''} = F^-_{W'} \circ F^-_{W} \colon \HFm(Y) \to \HFm(Y'').
\]
Finally, we have that $F^-_{I \times Y} = \mathrm{id}_{\HFm(Y)}$.

%% file: ZH-split_links.bbl
\begin{thebibliography}{DLVW22}

\bibitem[DLVW22]{DLVVW}
Aliakbar Daemi, Tye Lidman, David~Shea {{V}ela-Vick}, and C-M~Michael Wong.
\newblock Ribbon homology cobordisms.
\newblock {\em Adv. Math.}, 408:108580, 2022.

\bibitem[HKM23]{HKM}
Mark Hughes, Seungwon Kim, and Maggie Miller.
\newblock Non-isotopic splitting spheres for a split link in {$S^4$}.
\newblock ar{X}iv preprint ar{X}iv:2307.12140, 2023.

\bibitem[JM16]{concordance}
Andr{\'a}s Juh{\'a}sz and Marco Marengon.
\newblock Concordance maps in knot {F}loer homology.
\newblock {\em Geom. Topol.}, 20(6):3623--3673, 2016.

\bibitem[Kre01]{Kreck}
Matthias Kreck.
\newblock h--cobordisms between 1--connected 4--manifolds.
\newblock {\em Geometry \& Topology}, 5(1):1--6, 2001.

\bibitem[Lau74]{Laudenbach}
Francois Laudenbach.
\newblock {\em Topologie de la dimension trois: homotopie et isotopie},
  volume~12 of {\em Ast{\'e}risque}.
\newblock Soci{\'e}t{\'e} Math{\'e}matique de France (SMF), Paris, 1974.

\bibitem[Law88]{Lawson}
Terry Lawson.
\newblock h-cobordisms between simply connected 4-manifolds.
\newblock {\em Topology and its Applications}, 28(1):75--82, 1988.

\bibitem[LLP23]{LLP}
Adam~Simon Levine, Tye Lidman, and Lisa Piccirillo.
\newblock New constructions and invariants of closed exotic 4-manifolds.
\newblock ar{X}iv preprint ar{X}iv:2307.08130, 2023.

\bibitem[OS04a]{HF:p&a}
Peter Ozsv{\'a}th and Zolt{\'a}n Szab{\'o}.
\newblock Holomorphic disks and three-manifold invariants: properties and
  applications.
\newblock {\em Annals of Mathematics}, pages 1159--1245, 2004.

\bibitem[OS04b]{HF}
Peter~S. Ozsv{\'a}th and Zolt{\'a}n Szab{\'o}.
\newblock Holomorphic disks and topological invariants for closed
  three-manifolds.
\newblock {\em Ann. of Math. (2)}, 159(3):1027--1158, 2004.

\bibitem[Ron95]{Rong}
Yongwu Rong.
\newblock Degree one maps of {Seifert} manifolds and a note on {Seifert}
  volume.
\newblock {\em Topology Appl.}, 64(2):191--200, 1995.

\bibitem[Rus03]{Rustamov}
Raif Rustamov.
\newblock Calculation of {H}eegaard {F}loer homology for a class of {B}rieskorn
  spheres.
\newblock ar{X}iv preprint math/0312071, 2003.

\bibitem[Sma62]{h-cob}
Stephen Smale.
\newblock On the structure of manifolds.
\newblock {\em Amer. J. Math.}, 84(3):387--399, 1962.

\bibitem[Zee65]{Z-man}
E~Christopher Zeeman.
\newblock Twisting spun knots.
\newblock {\em Trans. Amer. Math. Soc.}, 115:471--495, 1965.

\bibitem[Zem19]{Zemke}
Ian Zemke.
\newblock Knot {F}loer homology obstructs ribbon concordance.
\newblock {\em Ann. Math. (2)}, 190(3):931--947, 2019.

\end{thebibliography}
